\documentclass{article}

\usepackage{amsmath, amsthm, amsfonts, amssymb, amscd, mathrsfs}

\usepackage{color} 

\theoremstyle{plain}
\newtheorem{Thm}{Theorem}[section]
\newtheorem*{Thm*}{Theorem}
\newtheorem{Prop}[Thm]{Proposition}
\newtheorem{Cor}[Thm]{Corollary}
\newtheorem{Lem}[Thm]{Lemma}

\theoremstyle{definition}
\newtheorem{Def}[Thm]{Definition}

\numberwithin{equation}{section}

\renewcommand{\P}{\mathcal{P}}
\newcommand{\Q}{\mathbb{Q}}
\newcommand{\C}{\mathbb{C}}
\newcommand{\Z}{\mathbb{Z}}
\newcommand{\R}{\mathbb{R}}
\newcommand{\QG}{\Q[G]}
\newcommand{\CG}{\C[G]}
\newcommand{\FG}{F[G]}
\newcommand{\FH}{F[H]}
\newcommand{\G}{\mathcal{G}}
\renewcommand{\H}{\mathcal{H}} 
\newcommand{\nn}{\mathcal{N}}
\renewcommand{\L}{\mathcal{L}}
\newcommand{\K}{\mathcal{K}}
\newcommand{\M}{\mathcal{M}}
\newcommand{\dsum}{\displaystyle\sum}
\newcommand{\bmid}{\;\big|\;}

\DeclareMathOperator{\Aut}{Aut}
\DeclareMathOperator{\Span}{Span}

\DeclareMathOperator*{\Inn}{Inn}
\DeclareMathOperator*{\Stab}{Stab}

\newcommand{\thmref}[1]{Theorem \ref{#1}}
\newcommand{\corref}[1]{Corollary \ref{#1}}
\newcommand{\lemref}[1]{Lemma \ref{#1}}
\newcommand{\propref}[1]{Proposition \ref{#1}}

\newcommand{\secref}[1]{Section \ref{#1}}

\allowdisplaybreaks[4]
\begin{document}

\title{Primitive Idempotents of Schur Rings}
\author{Andrew Misseldine\footnote{Andrew Misseldine, Department of Mathematics, Brigham Young University, Provo, UT, emisseldine@math.byu.edu}}

\maketitle

\begin{abstract}
In this paper, we explore the nature of central idempotents of Schur rings over finite groups. We introduce the concept of a lattice Schur ring and  explore properties of these kinds of Schur rings. In particular, the primitive, central idempotents of lattice Schur rings are completely determined. For a general Schur ring $S$, $S$ contains a maximal lattice Schur ring, whose central, primitive idempotents form a system of pairwise orthogonal, central idempotents in $S$. We show that if $S$ is a Schur ring with rational coefficients over a cyclic group, then these idempotents are always primitive and are spanned by the normal subgroups contained in $S$. Furthermore, a Wedderburn decomposition of Schur rings over cyclic groups is given. Some examples of Schur rings over non-cyclic groups will also be explored.


\end{abstract}


In 1950, Perlis and Walker \cite{Perlis} published the following result on the rational group algebra of a finite abelian group:

\begin{Thm*}\nonumber Let $\zeta_n = e^{2\pi i/n}\in \C$. Let $G$ be a finite abelian group of order $n$. Then 
\[\QG \cong \bigoplus_{d\mid n} a_d\Q(\zeta_d),\] where $a_d$ is the number of cyclic subgroups (or cyclic quotients) of $G$ of order $d$. In particular, if $G = Z_n$ is a cyclic group of order $n$, then 
\[\Q[Z_n] \cong \bigoplus_{d\mid n} \Q(\zeta_d).\]
 \end{Thm*}
 \noindent One consequence of the above decomposition is a  solution to the isomorphism problem of group rings over finite abelian groups with integer coefficients. Since then, several other results about rational group algebras have been published and the study of rational group algebras of a finite group is still an active field of representation theory. 

Of particular importance is the problem of finding the set of primitive central idempotents of the rational group algebra $\QG$. An element $\varepsilon$ of a ring $R$ is \emph{idempotent} if $\varepsilon^2 = \varepsilon$ and is \emph{central} in $R$ if $\varepsilon \in Z(R)$. In a semisimple ring such as $\QG$, all two-sided ideals are generated by a central idempotent. We say that a central idempotent is \emph{primitive} if it cannot be expressed as a sum of two nonzero central idempotents. A semisimple ring may be expressed as a direct sum of indecomposable two-sided ideals, called a \emph{Wedderburn decomposition}, each of which is principal and generated by a primitive central  idempotent. In this situation products of distinct ideals are trivial, and hence the primitive central  idempotents are pairwise orthogonal. Each central idempotent is a sum of primitive central idempotents. Thus, the primitive central idempotents are the atomic building blocks associated to the ideal structure of $\QG$. 

In the case of the complex group algebra, it is well known that the central primitive idempotents can be computed using the irreducible characters of the group. Averaging the Galois conjugates of each primitive central idempotent in $\CG$, the idempotents of the group algebra $F[G]$ can be computed for any subfield $F\subseteq \C$. In particular, the primitive central idempotents of $\QG$ can be computed in this way. Although this is possible using the characters, it is often computationally laborious to compute the central idempotents of $\QG$ by this method. Instead, character-free methods have been developed to compute these idempotents using the subgroups of $G$.

Character-free formulas for the primitive central idempotents of a finite abelian group algebra with rational coefficients are outlined in Chapter VII of \cite{Jespers96}, which we reproduce below in \propref{thm:AbelianIdempotents}. These formulas were later simplified and extended by Jespers, Leal, and Paques \cite{Jespers03} to finite nilpotent groups and by  Olivieri, del R\'{i}o, and Sim\'{o}n \cite{Rio04} to finite abelian-by-supersolvable groups. Other recent papers on the primitive central idempotents of $\QG$ include Olivieri and del R\'{i}o \cite{Rio03},  Broche and  del R\'{i}o \cite{Rio07}, Ferraz and Polcino Milies \cite{Milies07}, Van Gelder and Olteanu \cite{Gelder11}, Jespers, Olteanu, and del R\'{i}o \cite{Jespers12}, and Jespers, Olteanu, and Van Gelder \cite{Jespers13}.

The group algebra is a special example of a class of algebras called Schur rings. The Schur rings were originally developed by Schur and Wielandt in the first half of the 20th century. Schur rings were first used to study permutation groups, but in later decades applications of Schur rings have emerged in combinatorics, graph theory, and design theory \cite{KlinPoschel, Ma}.

The purpose of this paper is to extend the formulas for the primitive central idempotents of $\QG$ to all Schur rings over $G$, when $G$ is a finite cyclic group (\thmref{thm:ScycIdem}) and to extend the Wedderburn decomposition of $\QG$ given by Perlis and Walker to all Schur rings over a finite cyclic group (\thmref{thm:Sperlis}). We also provide examples of Schur rings over abelian groups where these formulas cannot be extended.  In group algebras, for each lattice of normal subgroups of a finite group, there corresponds a family of central idempotents. These lattices of normal subgroups naturally give rise to Schur rings. Furthermore, these systems of idempotents can often capture the primitive idempotents of related Schur rings. In the case of cyclic groups, it will be shown that the central, primitive idempotents of a Schur ring correspond to the lattice of $S$-subgroups.

In \secref{sec:Schur} Schur rings and their elementary properties are presented. This section also focuses on Cayley maps, these being maps on group algebras which are induced from group homomorphisms, and we also provide criteria for when the Cayley image of a Schur ring is also a Schur ring. In \secref{sec:IdemGroup}, lattices of normal subgroups of finite groups will be used to construct complete systems of orthogonal central idempotents in the group algebra and properties of these system of idempotents are explored. From this, the formula for the primitive central idempotents of lattice Schur rings is proven (\thmref{thm:LatticeIdem}). Similarly, in \secref{sec:IdemSchur}, lattices of normal $S$-subgroups are used to build complete systems of orthogonal idempotents in Schur rings. They are shown to be primitive idempotents when the group is cyclic.  In \secref{sec:DecompSchur}, a Wedderburn decomposition of Schur rings over cyclic groups is provided. \secref{sec:noncyclic} offers a few examples of Schur rings over abelian groups and considers their primitive idempotents.

Throughout let $G$ denote a finite group and $F$ a field whose characteristic does not divide $|G|$. For each $A\subseteq G$, let $\overline{A}$ denote the element $\sum_{g\in A} g$ in the group algebra $F[G]$. The cyclic group of order $n$ will be denoted by $Z_n$.


\section{Schur Rings}\label{sec:Schur}
\begin{Def} Let $\{C_1, C_2, \ldots, C_r\}$ be a partition of a finite group $G$ and let $S$ be the subspace of $F[G]$ spanned by $\overline{C_1}, \overline{C_2},\ldots \overline{C_r}$. We say that $S$ is a \emph{Schur Ring} over $F[G]$ if 
\begin{enumerate}
\item $C_1 = \{1\}$, 
\item For each $i$, there is a $j$ such that $(C_i)^{-1} = C_j$,
\item For each $i$ and $j$, $\overline{C_i}\cdot\overline{C_j} = \dsum_{k=1}^r \lambda_{ijk}\overline{C_k}$, for $\lambda_{ijk}\in F$.
\end{enumerate}
\end{Def}

\begin{Def} Let $S$ be a Schur ring over $F[G]$ and let $C\subseteq G$. We say that $C$ is an \emph{$S$-set} of $G$ if $\overline{C}\in S$. If $C$ is also a subgroup of $G$, then we say that $C$ is an \emph{$S$-subgroup} of $G$. If $C$ is one of the classes associated to the partition of $G$, then $C$ is called an \emph{$S$-class}.
\end{Def}

Every finite group algebra $F[G]$ is a Schur ring, resulting from the partition of singletons on $G$. The partition $\{\{1\}, G\setminus \{1\}\}$ affords a Schur ring, called the \emph{trivial Schur ring} of $G$.

Let $\H\le \Aut(G)$ and \[F[G]^\H = \{\alpha \in F[G]\mid \sigma(\alpha) = \alpha,\; \text{for all}\; \sigma\in \H\}.\]  Then $\FG^\H$ is a Schur ring afforded by the partition of $G$ corresponding to the orbits of the $\H$-action on $G$. These Schur rings are referred to as \emph{orbit Schur rings}. The center of $\FG$ is an orbit Schur ring with $\H=\Inn(G)$.

Let $S$ and $T$ be Schur rings over $F[G]$ and $F[H]$, respectively. We naturally can view $G$ and $H$ as subgroups of $G\times H$. The Schur rings $S$ and $T$ provide partitions of $G$ and $H$, respectively. Thus, we can create a partition on $G\times H$ by taking all possible products of $S$-classes and $T$-classes. This product partition affords a Schur ring $S\cdot T$, called the \emph{dot product} of $S$ and $T$. Furthermore, $S\cdot T \cong S\otimes_F T$, as $F$-algebras.

Let $1 < K \le H <G$  be a sequence of finite groups such that $K\trianglelefteq G$. Let $S$ be a Schur ring over $H$ and $T$ a Schur ring over $G/K$. Let $\pi : G \to G/K$ be the quotient map.  Let $\P$ be the partition of $G/K$ corresponding to $T$. Then $\pi^{-1}(\P)$ provides a partition of $G$ such that $K\in \pi^{-1}(\P)$. In particular, the classes in $\pi^{-1}(\P)$ are unions of cosets of $K$. Assume also that  $H/K$ is a $T$-subgroup, $K$ is an $S$-subgroup, and $\pi(S) = T\cap F[H/K]$, where the map $\pi$ is understood to be the linear extension of $\pi : G\to G/N$ to the group algebra $\FG$. Thus, the partitions corresponding to $S$ and $\pi^{-1}(\P)$ are compatible on their overlap, and we can refine the classes in $H$ from $\pi^{-1}(\P)$ using the partition associated to $S$. This constructs a Schur ring $S\wedge T$ over $G$ such that $(S\wedge T)\cap F[H] = S$ and $\pi(S\wedge T) = T$. A Schur ring of the form $S\wedge T$ is referred to as a \emph{wedge product} of Schur rings. Let $\pi^{-1}(T)$ denote the subalgebra of $F[G]$ afforded by the partition $\pi^{-1}(\P)$. Then $\pi^{-1}(T)$ is isomorphic to the Schur ring $T$ via the map $\pi : \pi^{-1}(T) \to T$ and is an ideal of $S\wedge T$. Viewing $S$ as a subalgebra of $S\wedge T$, we have $S\wedge T = S + \pi^{-1}(T)$. See \cite{LeungI} for further treatment of  wedge products.

\begin{Thm}[Leung and Man \cite{LeungI, LeungII}]\label{thm:LeungMan} Let $G = Z_n$ and let $S$ be a Schur ring over $G$. Then $S$ is trivial, an orbit ring, a dot product of Schur rings, or a wedge product of Schur rings. \end{Thm}

Define additional operations on $\FG$ as follows: 
\[* : \FG \to \FG :\qquad \left(\sum_{g\in G} \alpha_gg\right)^* = \sum_{g\in G} \alpha_gg^{-1}\] and the \emph{Hadamard product} \[\circ : \FG\times \FG \to \FG: \qquad \left(\sum_{g\in G} \alpha_gg\right) \circ \left(\sum_{g\in G}\beta_gg\right) = \sum_{g\in G} \alpha_g\beta_gg.\] Schur rings can then be characterized by these operations. 

\begin{Prop}[\cite{Muzychuk94} Lemma 1.3]\label{thm:circleProduct} Suppose that $S$ is a subalgebra of $F[G]$. Then $S$ is a Schur ring if and only if $S$ is closed under $*$ and $\circ$ and $1, \overline{G}\in S$.\end{Prop}

\begin{Def}\label{def:latticeSchur} Let $G$ be a finite group and $\L$ a sublattice of normal subgroups of $G$. Then we define \[S(\L) = \Span_F\{\overline{H} \mid H\in\L\}.\] Since $\overline{H}\circ\overline{K} = \overline{H\cap K}$ and $\overline{H}\cdot\overline{K} = |H\cap K|\overline{HK}$ for $H,K\in \L$, $S(\L)$ is a Schur ring, by \propref{thm:circleProduct}.  For this reason, $S(\L)$ will be called a \emph{lattice Schur ring}.
\end{Def}

For any finite group $G$, the trivial Schur ring is a lattice Schur ring, corresponding to the lattice $\{1, G\}$. It was observed by Muzychuk \cite{Muzychuk93} that for cyclic groups the lattice Schur rings correspond exactly with the rational Schur rings, those Schur rings which are fixed under all group automorphisms. 

It turns out that lattice Schur rings provide another way to construct Schur rings beyond the three methods used in the Leung and Man classification theorem. For example, let $G = Z_5\times Z_5 = \langle a, b\rangle$, let $\L = \{1, \langle a\rangle, \langle b\rangle, \langle ab\rangle, G\}$, and let $S = S(\L)$. Let $C = G \setminus (\langle a\rangle \cup \langle b\rangle \cup \langle ab\rangle)$, so that $|C| = 12$. Hence, $C$ is one of the $S$-classes. Since $C \neq G\setminus 1$, $S$ is not trivial. Likewise, $S$ cannot be a dot product of Schur rings since $C$ is not a product of two $S$-classes contained in proper subgroups of $G$. Also, $S$ cannot be a wedge product since $C$ is not a union of cosets for any nontrivial subgroup. If $S$ is an orbit Schur ring, it is generated by automorphisms such that $\langle a\rangle$, $\langle b\rangle,$ and $\langle ab\rangle$ are invariant subgroups. But the only automorphism subgroups with this property are cyclic and are generated by the identity map, inversion map, or by the squaring map. The partitions of $G$ corresponding to these automorphism groups are distinct from $S$, which implies that $S$ is not an orbit Schur ring. This example then shows that the Leung-Man classification theorem for cyclic groups cannot be extended to arbitrary abelian groups.

The following three properties of Schur rings are due to Wielandt \cite{Wielandt64}.

\begin{Prop}[\cite{Wielandt64} p. 56]\label{lem:sameCoeff} Let $G$ be a finite group and let $S$ be a Schur ring over $F[G]$. Let $\alpha\in S$ such that $\alpha = \sum_{g\in G} \alpha_gg$. Then $\{g\in G\mid \alpha_g = c\}$ is an $S$-set for each $c\in F$. 
\end{Prop}  

\begin{Prop}[\cite{Wielandt64} p. 58]\label{prop:Stab} Let   $S$ be a Schur ring over $G$. Let $\alpha\in S$ and let $\Stab(\alpha) = \{g\in G\mid \alpha g = \alpha\}$. Then $\Stab(\alpha)$ is an $S$-subgroup of $G$. \end{Prop}


\begin{Prop}[\cite{Wielandt64} p. 58]\label{prop:Ssubgroup} Let $S$ be a Schur ring over $G$ and let $C$ be an $S$-set. Then $\langle C\rangle$ is an $S$-subgroup.
\end{Prop}



\begin{Def}\label{def:cayley} Let $G$ and $H$ be groups and let $A$ and $B$ be subalgebras of $F[G]$ and $F[H]$, respectively. Let $f : A\to B$ be an $F$-algebra homomorphism. If $f$ is the restriction to $A$ of an induced group homomorphism $\varphi : G \to H$, then we say that $f$ is a \emph{Cayley homomorphism}.
\end{Def}

Let $\varphi : G \to H$ be a group homomorphism. Let $\varphi$ also denote its linear extension $\varphi : \FG \to \FH$. Let $g \in G$. Then \[\varphi(g^*) = \varphi(g^{-1}) = \varphi(g)^{-1} = \varphi(g)^*.\] By linearity, $\varphi(\alpha^*) = \varphi(\alpha)^*$ for all $\alpha\in \FG$. In particular, Cayley maps always preserve the involution structure of $\FG$. 

\begin{Lem}[\cite{Muzychuk94} Proposition 1.5]\label{lem:HadamardMap} Let $\varphi : G \to H$ be a group homomorphism with $\ker \varphi = K$. Let $\alpha, \beta \in \FG$. Then 
\[\varphi(\alpha) \circ \varphi(\beta) = \frac{1}{|K|}\varphi( (\alpha\cdot\overline{K}) \circ (\beta\cdot\overline{K})).\]
\end{Lem}

In Muzychuk's original proof, he proves \lemref{lem:HadamardMap} under the assumption that $G$ is an abelian group. Since $K\trianglelefteq G$, $\overline{K}$ is central in $\FG$, and Muzychuk's proof remains valid without the abelian assumption.


\begin{Cor}\label{cor:HadamardMap} Let $\varphi : G \to H$ be a group homomorphism with $\ker \varphi = K$. Let $S$ be a Schur ring over $\FG$ such that $\overline{K} \in S$. Then $\varphi(S)$ is a Schur ring over a subgroup of $H$. Furthermore, if $\varphi$ is surjective, then $\varphi(S)$ is a Schur ring over $H$.
\end{Cor}

\begin{proof}
$\varphi(S)$ is always a subalgebra of $\FH$ closed under $*$ and contains $1, \overline{\varphi(G)}$. By \lemref{lem:HadamardMap}, $\varphi(S)$ is closed under $\circ$, and hence $\varphi(S)$ is a Schur ring over $\varphi(G)$ by \propref{thm:circleProduct}.
\end{proof}

\corref{cor:HadamardMap} was originally proved by Leung and Ma \cite{Leung90} using a different proof.

As stated earlier, for any lattice $\L$ of normal subgroups of a finite group $G$, $S(\L)$ is a Schur ring over $G$ spanned by the elements of $\L$. Let $H\trianglelefteq G$ and let $\varphi : G \to G/H$ be the quotient map. Suppose that $\L$ is a distributive lattice. Then we claim that $\varphi(S(\L))$ is a Schur ring over $G/H$, even if $H\notin \L$. As in the proof of \corref{cor:HadamardMap}, it suffices to show that $\varphi(S(\L))$ is closed under $\circ$. If $K_1, K_2\in \L$, then 
\begin{eqnarray*}
\overline{K_1H} \circ \overline{K_2H} &=& \overline{K_1H\cap K_2H} = \overline{(K_1\cap K_2)H} = \dfrac{1}{|(K_1\cap K_2)\cap H|}\overline{K_1\cap K_2}\cdot \overline{H}\\
& =& \dfrac{1}{|K_1\cap K_2\cap H|}(\overline{K_1}\circ\overline{K_2})\cdot \overline{H},
\end{eqnarray*}
 where the second equality holds by the distributivity of the lattice. Then 
\begin{eqnarray*}
\varphi(\overline{K_1}) \circ \varphi(\overline{K_2}) &=& \frac{1}{|H|}\varphi( (\overline{K_1}\cdot\overline{H}) \circ (\overline{K_2}\cdot\overline{H})), \quad \text{by \lemref{lem:HadamardMap}},\\
&=& \dfrac{|K_1\cap H||K_2\cap H|}{|H|}\varphi(\overline{K_1H}\circ \overline{K_2H})\\
& =& \dfrac{|K_1\cap H||K_2\cap H|}{|H||K_1\cap K_2\cap H|}\varphi((\overline{K_1}\circ \overline{K_2})\cdot\overline{H})\\
&=& |(K_1\cap H)(K_2\cap H)|\varphi(\overline{K_1}\circ\overline{K_2}) \in \varphi(S(\L)).
\end{eqnarray*} Therefore, $\varphi(S(\L))$ is closed under $\circ$, which proves the claim.  This fact is reported in the next proposition.

\begin{Prop}\label{prop:latticeQuotient} Let $G$ be a finite group and let $\L$ be a distributive lattice of normal subgroups of $G$. Let $\varphi : G\to H$ be a group homomorphism. Then $\varphi(S(\L))$ is a lattice Schur ring over a subgroup of $H$. 
\end{Prop}

Let $G$ be a finite cyclic group. Then the lattice of subgroups of $G$ is distributive, and hence any sublattice is also distributive. Thus, $\varphi(S(\L))$ is a Schur ring for any homomorphism $\varphi$ and any lattice $\L$ of subgroups of $G$. The next theorem generalizes this for any Schur ring over a cyclic group.

\begin{Thm}\label{thm:CayleySchur} Let $G$ be a finite cyclic group and $S$ be a Schur ring over $\FG$. If $\varphi : G \to L$ is a group homomorphism, then $\varphi(S)$ is a Schur ring over a subgroup of $L$.
\end{Thm}

\begin{proof}
Let $S$ be a Schur ring over $G = Z_n$. We proceed by induction on $|G|$. If $|G| = p$, a prime, then the only normal subgroups are $1$ and $G$, which are necessarily $S$-subgroups. Thus, the property holds for $|G| = p$, by \corref{cor:HadamardMap}.

Suppose now the property holds for all proper divisors of the integer $n$ and let $S$ be a Schur ring over $G=Z_n$. By \thmref{thm:LeungMan}, $S$ is a trivial, orbit, dot product, or wedge product Schur ring. If $S$ is trivial, then it is a lattice Schur ring.  So, $\varphi(S)$ is a Schur ring by \propref{prop:latticeQuotient}. If $S$ is an orbit Schur ring, then every subgroup of $G$ is an $S$-subgroup since every subgroup is characteristic. Thus, $\varphi(S)$ is a Schur ring by \corref{cor:HadamardMap}. If $S = R\cdot T$ for Schur rings $R$ and $T$ over subgroups $H$ and $K$, respectively, such that $G = H\times K$, then $\varphi(S) = \varphi(R\cdot T) = \varphi(R)\cdot\varphi(T)$. Since $\varphi(R)$ and $\varphi(T)$ are Schur rings by induction, $\varphi(S)$ is the dot product of Schur rings and hence a Schur ring itself. Lastly, let $S = R\wedge T$ for Schur rings $R$ and $T$ over normal subgroup $H$ and quotient group $G/K$, respectively. Let $\pi : G\to G/K$ be the quotient map. Then $S = R\wedge T = R + \pi^{-1}(T)$. Without the loss of generality, we may assume that $\varphi$ is the quotient map $\varphi : G \to G/N$. We likewise define $\pi^* : G/N \to G/KN$ and $\varphi^* : G/K \to G/KN$ to be quotient maps. Then it holds that $\varphi(\pi^{-1}(T)) = (\pi^*)^{-1}(\varphi^*(T))$. By induction, $\varphi(R)$ and $\varphi^*(T)$ are Schur rings. Therefore, $\varphi(S) = \varphi(R\wedge T) = \varphi(R) + \varphi(\pi^{-1}(T)) = \varphi(R) + (\pi^*)^{-1}(\varphi^*(T)) = \varphi(R) \wedge \varphi^*(T)$, which is a Schur ring.
This then proves the result for arbitrary $n$.
\end{proof}

The Cayley image of a Schur ring need not be a Schur ring. In fact, it is false even for Schur rings over abelian groups. Let $G = Z_2 \times Z_6 = \langle a, b\rangle$ and let $S = \Span_\Q\{1, b^3, b^2+b^4, b+b^5, a+ab^3, ab+ab^2, ab^4+ab^5\}$. Then $S$ is an orbit Schur ring afforded by the  subgroup generated by the automorphism $\sigma : a\mapsto ab^3, b\mapsto b^{-1}$. Let $\varphi : G \to Z_6$ be the projection homomorphism onto the subgroup $\langle b\rangle$, that is, $\pi : a\mapsto 1, b\mapsto b$. Then 
\begin{eqnarray*}
\varphi(S) &=& \Span_\Q\{1, b^3, b^2+b^4, b+b^5, 1+b^3, b+b^2, b^4+b^5\}\\
&=& \Span_\Q\{1, b^3, b^2+b^4, b+b^5, b+b^2\}.
\end{eqnarray*}  If $\varphi(S)$ were a Schur ring, then $(b+b^5)\circ (b+b^2) = b\in \varphi(S)$. Since $b\in \varphi(S)$, this implies that $\varphi(S) = \Q[Z_6]$, which is six-dimensional. But $\dim \varphi(S) \le 5$, which proves that $\varphi(S)$ is not a Schur ring. 


\section{Central Idempotents in Group Algebras}\label{sec:IdemGroup}

If the sum of a  set of idempotents is 1, we say that the set of idempotents is \emph{complete}. In particular, the set of all primitive central  idempotents is always complete in a semisimple ring. Furthermore, every central idempotent of the semisimple ring is a sum of primitive central  idempotents, and the primitive central  idempotents \emph{involved} in this sum are precisely the ones whose product with the idempotent is nonzero. 

Let $G$ be a finite group and $H\le G$.  Then for all $h\in H$, $h\overline{H} = \overline{H}h = \overline{H}$. Let $\widehat{H} = \dfrac{1}{|H|}\overline{H} \in \FG$. Then $\widehat{H}$ is an idempotent in $\FG$. If $H\trianglelefteq G$, then $\widehat{H}$ is a central idempotent. Note, $(\widehat{G})$ is a one-dimensional ideal in $\FG$, which implies that $\widehat{G}$ is always primitive in $\FG$. On the other hand, if $H\lneq G$,  then $\widehat{H}$ is not primitive in $\FG$ since $\widehat{H} = \widehat{G} + (\widehat{H}-\widehat{G})$.

Given any subgroups $H$ and $K$ of $G$, we have
\[\widehat{H}\widehat{K} = \dfrac{1}{|H||K|}\overline{H}\cdot\overline{K} = \dfrac{|H\cap K|}{|H||K|}\overline{HK} = \widehat{HK}.\] If $H$ is normal, then $HK \le G$ and $\widehat{HK}$ is an idempotent of $\FG$. If $H,K\trianglelefteq G$, then $HK$ is also normal in $G$. So, $\widehat{HK}$ is central in $\FG$.

 Let $\L$ be a \emph{semi-lattice of normal subgroups} of $G$, by which we mean $\L$ is a set of normal subgroups of $G$ which is closed under joins and contains $1$ and $G$. Thus, a lattice of normal subgroups  is a semi-lattice closed under intersections. For $H, K\in \L$, we say that $K$ \emph{covers} $H$ if $H < K$ and for all $L\in \L$ such that $H\le L\le K$, either $L=H$ or $L = K$. Let $\M(\L,H)$ denote the set of all covers of $H$ in the semi-lattice $\L$. When $\L$ is the whole lattice of normal subgroups of $G$, let $\M(G,H) = \M(\L,H)$.

For every semi-lattice of normal subgroups of $G$, there is an associated system of idempotents in $\FG$ as follows: let \[\varepsilon(\L,H) = \prod_{M\in \M(\L,H)} (\widehat{H} - \widehat{M}) \in \FG.\] Since each subgroup $M$ is normal, $\widehat{M}$ is central in $\FG$ and hence the order of the product is irrelevant and $\varepsilon(\L,H)$ is central in $\FG$.  Because $\M(\L,G) = \emptyset$,  let $\varepsilon(\L,G) = \widehat{G}$. When $\L$ is the whole semi-lattice, we let $\varepsilon(G,H) = \varepsilon(\L,H)$. This agrees with the notation introduced in \cite{Jespers03}.

\begin{Lem}\label{lem:idemidem} Let $\L$ be a semi-lattice of normal subgroups of $G$ with $H,K\in \L$. 
\begin{enumerate}
\item If $K \le H$, then $\widehat{K}\varepsilon(\L,H) = \varepsilon(\L,H).$
\item If $H < K$, then $\widehat{K}\varepsilon(\L,H) = 0$.
\item If $K \not\le H$, then $\widehat{K}\varepsilon(\L,H)=0$.
\end{enumerate}

In particular, $\varepsilon(\L,H)^2 = \varepsilon(\L,H)$ and  $\varepsilon(\L,H)\varepsilon(\L,K) = 0$ if $H\neq K$.
\end{Lem}

\begin{proof}
Clearly, $\widehat{K}\widehat{G} = \widehat{G}$, which shows the first condition for $H=G$. The remaining conditions are vacuously true for $H=G$, so we may assume that $H\neq G$. 
\begin{enumerate}
\item\label{item:idemidem1} For any $M\in \M(\L,H)$, $K \le H < M$. Thus, $\widehat{K}\widehat{H} = \widehat{H}$ and $\widehat{K}\widehat{M} = \widehat{M}$, which implies that $\widehat{K}(\widehat{H}-\widehat{M}) = \widehat{H}-\widehat{M}$. So, $\widehat{K}\varepsilon(\L,H) = \varepsilon(\L,H)$.

\item\label{item:idemidem2} Since $H < K$, there exists some cover $M$ of $H$ in $\L$ such that $H < M \le K$. Then $\widehat{K}(\widehat{H}-\widehat{M}) = \widehat{K} - \widehat{K} = 0$. This implies that $\widehat{K}\varepsilon(\L,H) = 0$.

\item Lastly, $\widehat{K}\varepsilon(\L,H) = \widehat{K}(\widehat{H}\varepsilon(\L,H)) = \widehat{KH}\varepsilon(\L,H) = 0$, where the first equality is by \eqref{item:idemidem1} and the third equality is by \eqref{item:idemidem2}. \qedhere
\end{enumerate}
\end{proof}

\begin{Prop}\label{prop:S=T} Let $\L$ be a semi-lattice of normal subgroups of $G$.  Then $S(\L)\footnote[2]{Although $S(\L)$ was originally defined for lattices, its definition naturally extends to the case when $\L$ is a semi-lattice. In this case, $S(\L)$ may not be a Schur ring, but it will be an algebra.}  = \Span_F\{\varepsilon(\L,H) \mid H\in \L\}$. In particular, \[\sum_{H\in \L} \varepsilon(\L,H) = 1.\]
\end{Prop}

\begin{proof}
Let $S = \Span\{\overline{H} \mid H\in \L\}$ and $T = \Span\{\varepsilon(\L,H) \mid H\in \L\}$. We must prove that $S=T$. Clearly, $T \subseteq S$. We prove the reverse containment by induction on $|\L|$. 

If $|\L| = 2$, then $\L = \{1, G\}$, $\varepsilon(\L,1) = \widehat{G}-1$, and $\varepsilon(\L,G) = \widehat{G}$. Hence, $S = T$. Suppose the claim holds for any semi-lattice with order less that $|\L|$. Let $H\in \L$ and let $\lceil H\rceil = \{K\mid K\in \L, H\le K\}$. Now, $\lceil H\rceil$ is a semi-lattice of normal subgroups of $G$ with unit $H$ and if $H\neq 1$ then $\lceil H\rceil$ has order strictly less $|\L|$. Let $\pi : G \to G/H$ be the quotient map. Then $\pi(\lceil H\rceil)$ is a semi-lattice of normal subgroups of $G/H$ and $\Span\{\overline{K/H} \mid K \in \lceil H\rceil\} = \Span\{\varepsilon(\pi(\lceil H\rceil),K/H) \mid K \in \lceil H\rceil\}$ by induction. Lifting this back to $G$, we have 
\begin{eqnarray*}
\Span\{\overline{K} \mid K \in \lceil H\rceil\} &=& \Span\{\varepsilon(\lceil H\rceil,K) \mid K \in \lceil H\rceil\}\\
& =& \Span\{\varepsilon(\L,K) \mid K \in \lceil H\rceil\} \subseteq T.
\end{eqnarray*} In particular, $\overline{H} \in T$ for all $H\neq 1$. But $\varepsilon(\L,1) = 1 + \alpha$, where $\alpha\in \Span\{\overline{H}\mid H\in \L, H\neq 1\} \subseteq T$. Therefore, $1 = \varepsilon(\L,1) - \alpha \in T$, which proves the claim $S=T$. 

By orthogonality, $\{\varepsilon(\L,H) \neq 0\mid H\in \L\}$ is a basis of $S$ and $S$ is semisimple. Hence, $\{\varepsilon(\L,H)\neq 0\mid H\in \L\}$ is the complete set of primitive central idempotents of $S$. Therefore, $\sum_{H\in \L} \varepsilon(\L,H) = 1$.
\end{proof}

In particular, $\{\varepsilon(\L,H) \mid H \in \L\}$ is a complete set of orthogonal central idempotents in $\FG$. 

\begin{Thm}\label{thm:LatticeIdem} Let $G$ be a finite group and let $F$ be a field with characteristic not dividing $|G|$. Let $\L$ be a semi-lattice of normal subgroups of $G$. Then $\{\varepsilon(\L,H) \neq 0 \mid H\in \L\}$ is a complete set of primitive central idempotents of $S(\L)$ and $S(\L) \cong \bigoplus_n F$, where $n = | \{\varepsilon(\L,H) \neq 0 \mid H\in \L\}|$.
\end{Thm}
\begin{proof}
Let $S= S(\L)$. By \propref{prop:S=T}, $S = \Span\{\varepsilon(\L,H)\mid H\in \L\}$ and $\{\varepsilon(\L,H)\neq 0\mid H\in \L\}$ is a basis of $S$. Thus, this basis must be a complete set of idempotents and the ideal of each idempotent must have dimension 1. Thus, each idempotent is primitive. 
\end{proof}

\begin{Cor} Let $S$ be a lattice Schur ring over $\FG$ and let $\varepsilon \in S$ be an idempotent. Then $\varepsilon\in \Span_F\{\overline{H} \mid \overline{H} \in S\}$.\end{Cor}

\begin{Lem}\label{thm:IdempotentSum} For any semi-lattice $\L$ of normal subgroups of a finite group $G$ and any $H\in \L$, we have
\[\varepsilon(\L,H) = \sum_{K} \varepsilon(G,K)\] where the sum ranges over all subgroups $K$ of $G$ such that $H\le K$ and $M\not\le K$ for all $M\in \M(\L,H)$. 
\end{Lem}

\begin{proof} Since $H\le K$ and $M \not\le K$ for all $M\in \M(\L,H)$, $\widehat{H}\varepsilon(G,K) = \varepsilon(G,K)$ and $\widehat{M}\varepsilon(G,K) = 0$. Hence, $\varepsilon(\L,H)\varepsilon(G,K) = \varepsilon(G,K)$. Thus, all of the primitive idempotents involved in $\varepsilon(G,K)$ are also involved in $\varepsilon(\L,H)$. Furthermore, $\{\varepsilon(G,K)\mid K\trianglelefteq G\}$ is a compete set of orthogonal central idempotents. Therefore, every primitive central idempotent of $\FG$ is involved with one and only one of the $\varepsilon(G,K)$. This  determines all of the primitive central idempotents involved in $\varepsilon(\L,H)$. By partitioning these primitive central idempotents, we get the desired  equality.
\end{proof}

The previous lemma then shows how the system of central idempotents resulting from $\L$ can be decomposed into a sum of idempotents of the form $\varepsilon(G,H)$ in $\FG$. We note however that $\varepsilon(G,H)$ is not necessarily primitive. In fact, $\varepsilon(G,H)$ may be zero. For example, let $G = Z_2\times Z_2$ and $F = \Q$. Then $\varepsilon(G,1) = 0$. On the other hand, when $G$ is cyclic, $\varepsilon(\L,H) \neq 0$ for all $H\in \L$, as we now show.

\begin{Lem}\label{lem:CycIdem} Let $G$ be a finite cyclic group and let $\L$ be a semi-lattice of subgroups of $G$. Then $\varepsilon(\L,H)\neq 0$ for all $H\in \L$. \end{Lem}

\begin{proof}
For a cyclic group $G$, $S(\L)$ has for a basis the set $\{\overline{H}\mid H\in \L\}$. This can be seen by examining the generators of each subgroup in $\L$. Thus, $\dim_F S(\L) = |\L|$. By \propref{prop:S=T}, we have that $S(\L) = \Span\{\varepsilon(\L,H)\mid H\in \L\}$. Therefore, 
\[|\L| = \dim S(\L) = |\{\varepsilon(\L,H) \neq 0\mid H\in \L\}| \le |\{\varepsilon(\L,H)\mid H\in \L\}| \le |\L|.\] Therefore, $|\{\varepsilon(\L,H) \neq 0\mid H\in \L\}| = |\{\varepsilon(\L,H)\mid H\in \L\}|$, which implies that $\varepsilon(\L,H)\neq 0$ for all $H\in \L$.
\end{proof}

On the other hand, let $G = Z_2\times Z_2 = \langle a,b\rangle$ and let $\L = \{1, \langle a\rangle, G\}$. Although $G/1$ is not cyclic, $\varepsilon(\L,1) = 1 - \widehat{\langle a\rangle} = \frac{1}{2} (1-a) \neq 0$. Thus, for general semi-lattices of abelian groups, $\varepsilon(\L,H)$ can be nonzero even if $G/H$ is not cyclic. In this example $\varepsilon(\L,1)$ is imprimitive in $\QG$ since $\varepsilon(\L,1) = \varepsilon(G,\langle b\rangle) + \varepsilon(G,\langle ab\rangle)$.

Let $\L$ be a semi-lattice of normal subgroups of a finite group $G$. For any $H\in \L$, let
\[\nn(\L,H) = \{K\trianglelefteq G \mid H \le K\;\text{and}\; M\not\le K,\;\text{for all}\; M\in \M(\L,H)\}.\] So, $\nn(\L,H)$ is the set of all normal subgroups between $H$ and an $\L$-cover of $H$. Then \[\varepsilon(\L,H) = \sum_{K\in \nn(\L,H)} \varepsilon(G,K),\] by \lemref{thm:IdempotentSum}.  Generalizing the above set, for any $N\in\nn(\L,H)$,
\[\nn(\L,H,N) = \{K \in \nn(\L,H)\mid K \ge N\}.\]  We mention that $\nn(\L,H)$ is closed under intersections, as is $\nn(\L,H,N)$.

\begin{Lem}\label{thm:IdemDescent} Let $\L$ be a semi-lattice of normal subgroups of a finite group $G$ and let $H\in \L$. Let $N\in \nn(\L,H)$ and let $\pi : G \to G/N$ be the natural quotient map. Then $\pi$ induces a bijection $\nn(\L,H,N) \to \nn(\pi(\L), N/N)$.\end{Lem}

\begin{proof}
First, let $K\in \nn(\L,H,N)$.  So, $\pi(K)$ is normal in $G/N$ and clearly $N/N \le \pi(K)$. Suppose $M'\in \pi(\L)$ such that $N/N\le M'\le \pi(K)$. Then there exists some $M\in \L$ such that $\pi(M) = M'$. Since $MH\in \L$ and $\pi(MH) = M'$, we may assume that $H\le M$. Next,
\[H \le M \le MN \le KN \le K.\] Since $K\in \nn(\L,H,N)$, the only normal subgroup between $K$ and $H$ contained in $\L$ is $H$. Thus, $M=H$, which implies \[M'=\pi(M)=\pi(H) = N/N.\] Since there are no subgroups in $\pi(\L)$ between $N/N$ and $\pi(K)$  other than $N/N$ itself, $\pi(K) \in \nn(\pi(\L),N/N)$. Hence, $\pi(\nn(\L,H,N))\subseteq \nn(\pi(\L),N/N).$

Second, suppose $\pi(K_1) = \pi(K_2)$, for $K_1, K_2\in \nn(\L,H,N)$. Since $N\le K_1\cap K_2$, $K_1=K_2$, by correspondence. Therefore, $\pi :  \nn(\L,H,N) \to \nn(\pi(\L), N/N)$ is injective.

Lastly, let $K'\in \nn(\pi(\L),N/N)$. Then there exists a unique normal subgroup $K$ of $G$ such that $\pi(K)= K'$ and $N\le K$. Let $L\in \L$ such that $H\le L\le K$. Then $N/N\le \pi(L) \le K'$. Since $\pi(L)\in \pi(\L)$, it must be that $\pi(L) = N/N$, which implies that $L \le N$. But $N\in \nn(\L,H)$. Thus, $L = H$, which proves that $K \in \nn(\L,H,N)$, also. This shows that $\pi :  \nn(\L,H,N) \to \nn(\pi(\L), N/N)$ is surjective.
\end{proof} 

\begin{Cor}\label{cor:IdemDescent} Let $G$ be a finite group with semi-lattice of normal subgroups $\L$. Let $H\in \L$ and $N\in \nn(\L,H)$. Furthermore, if $\pi : G\to G/N$ is the quotient map, then $\pi(\varepsilon(\L,H)) = \varepsilon(\pi(\L), N/N)$.\end{Cor}

\begin{proof}
By \lemref{thm:IdempotentSum}, 
\[\varepsilon(\L,H) = \sum_{K\in \nn(\L,H,N)} \varepsilon(G,K) + \sum_{L\in \nn(\L,H) \setminus \nn(\L,H,N)} \varepsilon(G,L).\] For each $L \not\ge N$, $\pi(\varepsilon(G,L)) = 0$, and for each $K\ge N$, $\pi(\varepsilon(G,K)) = \varepsilon(G/N, K/N)$. Thus, 
\begin{eqnarray*}
\pi(\varepsilon(\L,H)) &=& \sum_{K\in \nn(\L,H,N)} \varepsilon(G/N, K/N) = \sum_{K/N\in \nn(\pi(\L),N/N)} \varepsilon(G/N,K/N)\\ 
&=& \varepsilon(\pi(\L),N/N),
\end{eqnarray*} where the second equality follows by \lemref{thm:IdemDescent} and the third follows by \lemref{thm:IdempotentSum}.
\end{proof}

The primitivity of $\varepsilon(G,H)$ can be determined when $G$ is abelian.

\begin{Prop}[\cite{Jespers03} Corollary 2.1]\label{thm:AbelianIdempotents} The set $\{\varepsilon(G,H) \mid H\le G, $G/H$\; \text{is cyclic}\}$ is a complete set of primitive central idempotents in $\QG$ when $G$ is abelian. If $\varepsilon \in \QG$ is any idempotent, then $\varepsilon\in \Span_\Q\{\overline{H} \mid H \le G\}$. \end{Prop}

\begin{Cor}\label{thm:CyclicIdempotents} The set $\{\varepsilon(G,H) \mid H\le G\}$ is a complete set of primitive central idempotents in $\QG$ when $G$ is cyclic. \end{Cor}

Removing zero idempotents if necessary, \lemref{thm:IdempotentSum} gives a decomposition of $\varepsilon(\L,H)$ into primitive idempotents in $\QG$ when $G$ is abelian.

Let $G$ be an abelian group. Suppose that $H\le G$ but $G/H$ is not cyclic. By \propref{thm:AbelianIdempotents}, $\varepsilon(G,H)$ is not a primitive idempotent but $\varepsilon(G,H)\varepsilon(G,K) = 0$ for all $G/K$ cyclic. Thus, $\varepsilon(G,H) = 0$, that is, if $G$ is abelian, $\varepsilon(G,H) \neq 0$ if and only if $\varepsilon(G,H)$ is primitive if and only if $G/H$ is cyclic.


\section{Primitive Idempotents of Schur Rings over Cyclic Groups}\label{sec:IdemSchur}

Let $S$ be a Schur ring over $\FG$, for some finite group $G$, not necessarily abelian. Let $H, K$ be normal $S$-subgroups of $G$. Then $\overline{H}\cdot\overline{K} = |H\cap K|\overline{HK} \in S$ and $\overline{H}\circ \overline{K} = \overline{H\cap K} \in S$. Thus, the collection of all normal $S$-subgroups $\L$ forms a lattice of normal subgroups of $G$. As shown above, associated to this lattice is a complete set of idempotents in $\FG$. Let $\varepsilon(S,H) = \varepsilon(\L,H)$. Since $\L$ is a lattice, $S(\L)$ is a lattice Schur ring contained in $S$ and is maximal with respect to being the largest lattice subring in $S$. Furthermore, $S(\L) = \Span_F\{\varepsilon(S,H)\mid \overline{H}\}$, and hence contains many of the central idempotents of $S$. Under some conditions, $S(\L)$ contains all the central idempotents of $S$, for example when $S = S(\L)$. We will see in this section that if $G = Z_n$ and $F=\Q$ then $S(\L)$ contains all the idempotents of $S$. 

Let $n$ be a positive integer with prime factorization given as \[n = \prod_{i=1}^r p_i^{a_i},\] where each $p_i$ is a distinct prime. Set \[\lambda(n) = (-1)^{\sum_{i=1}^r a_i},\] and \[\text{Id}(n) = n.\] It is elementary to check that $\lambda$ and $\text{Id}$ are multiplicative functions\footnote[3]{A function $f : \Z^+ \to \R$ is multiplicative if $f(1) = 1$ and $f(mn) = f(m)f(n)$ whenever $\gcd(m,n) = 1$.}. Let $\beta$ be the Dirichlet convolution of $\lambda$ and $\text{Id}$, that is, 
\[\beta(n) = (\lambda\;\sharp\;\text{Id})(n) = \sum_{d\mid n}\lambda(d)(n/d).\] $\beta$ is the alternating-sum-of-divisors function. Since the convolution of multiplicative functions is multiplicative, we have that $\beta$ is also a multiplicative function.  A detailed treatment of $\beta$ can be found in \cite{Toth}.

Let $G = Z_n$ be a cyclic group of order $n$. For each divisor $d\mid n$, let $L_d$ be the set of elements of order $d$ in $G$. Since $G$ has a unique subgroup of order $d$, which is necessarily cyclic, we will refer to this subgroup as $G_d$. Thus, $L_d$ is the set of generators of $G_d$ and is referred to as the \emph{$d$th layer} of $G$. 

Consider the expansion 
\begin{eqnarray}
\varepsilon(G,1) &=& \prod_{i=1}^r (1-\widehat{G_{p_i}}) \nonumber\\ 
&=& \label{eq:beta1}1 - \sum_{i \le r} \widehat{G_{p_i}} + \sum_{i<j\le r} \widehat{G_{p_ip_j}} - \sum_{i<j<k\le r} \widehat{G_{p_ip_jp_k}} + \ldots \pm \widehat{G_m},\\
&=& \label{eq:beta2} \sum_{d\mid m} c_d\overline{L_d},
\end{eqnarray} where $m = \prod_{i=1}^r p_i$. Let $a\mid m$. By comparing coefficients in \eqref{eq:beta1} and \eqref{eq:beta2}, we have 
\begin{eqnarray}
c_a &=& \sum_{a\mid d\mid m} \dfrac{\lambda(d)}{d} = \lambda(a)\sum_{a\mid d\mid m} \dfrac{\lambda(d/a)}{d}\nonumber\\
&=&  \dfrac{\lambda(a)}{m}\sum_{a\mid d\mid m} \lambda(d/a)(m/d) = \dfrac{\lambda(a)}{m}\sum_{d'\mid (m/a)} \lambda(d')((m/a)/d')\nonumber\\
&=&\label{eq:beta3} \dfrac{\lambda(a)\beta(m/a)}{m}.
\end{eqnarray}  Also $c_a=0$ for any  $a\nmid m$. 

Next, let $H \trianglelefteq G$. Then for all $h\in H$, $h\varepsilon(G,H) = \varepsilon(G,H)$. Thus, the coefficients of $\varepsilon(G,H)$ are constant over cosets of $H$. Let $\pi : G \to G/H$ be the natural quotient map. Then, as seen above, $\varphi(\varepsilon(G,H)) = \varepsilon(G/H,H/H)$. Let $a\mid n$ and let $g\in G$ be an element of order $a$. Then $\varphi(g) \in G/H$ is an element of order $a' = \dfrac{a}{\gcd(a,|H|)}$.  Finally, let $c_a$ be the coefficient of $g$ in $\varepsilon(G,H)$, let $c'_a$ be the coefficient of $\varphi(g)$ in $\varepsilon(G/H, H/H)$, and let $m'$ be the product of distinct prime divisors of $n/|H|$. Thus, by \eqref{eq:beta3},
\begin{equation} c_a = \dfrac{1}{|H|}c'_a = \dfrac{1}{|H|}\left(\dfrac{\lambda(a')\beta(m'/a')}{m'}\right)\end{equation}

Thus, combining the above formula with \lemref{thm:IdempotentSum}, it is possible to compute the coefficients of $\varepsilon(\L,H)$ for any lattice of subgroups of $G = Z_n$.

Now, in a Schur ring, \propref{lem:sameCoeff} applies and by examining coefficients of $\varepsilon(G,1)$ certain $S$-subgroups can be identified. 

\begin{Lem}\label{lem:ScycIdemTrivial} Let $S$ be a Schur ring over $\QG$ where $G$ is a finite cyclic group. If $\varepsilon(G,1) \in S$, then $\overline{G_p}\in S$ for all $p\bmid |G|$. \end{Lem}

\begin{proof} 

Suppose that $|G| = n = \prod_{i=1}^r p_i^{a_i}$ is a prime factorization. Let $m = \prod_{i=1}^r p_i$. By \propref{prop:Ssubgroup}, if $\overline{L_p}\in S$ then $\overline{G_p}\in S$. So it suffices to show that $\overline{L_p}\in S$ for all $p\mid n$. 

As noted in \eqref{eq:beta3}, the coefficient of $\overline{L_p}$ in $\varepsilon(G,1)$ is $\dfrac{\lambda(p)\beta(m/p)}{m}$ for all $p\mid m$. Suppose that for some other divisor $d\mid m$, 
\begin{equation}\label{eq:beta}\dfrac{\lambda(p)\beta(m/p)}{m} = \dfrac{\lambda(d)\beta(m/d)}{m}.\end{equation} Then $\beta(m/p) = \beta(m/d)$, since $\beta(k) > 0$ for all positive $k$. Since $m$ is square-free and $\beta$ is multiplicative, this implies that $\beta(p) = \beta(d)$. 

Suppose $d = \prod_{i=1}^s q_i$, where each $q_i$ is a prime divisor of $m$. Then $\beta(p) = p-1$ and 
\[\beta(d) = \beta\left(\prod_{i=1}^s q_i\right) = \prod_{i=1}^s (q_i-1).\] Now, if $p\mid d$, then $q_k= p$ for some $1\le k\le s$ and $\prod_{i=1}^s(q_i-1) = q_k-1 \Rightarrow (q_1-1)\ldots \widehat{(q_k-1)}\ldots(q_s-1) = 1$, where here \; $\widehat{}$\; denotes an omitted factor. That implies that $d=p$ or $d=2p$. But if $d=2p$, then $\lambda(d) = 1$, while $\lambda(p) = -1$, which contradicts \eqref{eq:beta}. Therefore, we may assume that $\gcd(p,d) = 1$. Furthermore, since $\prod_{i=1}^s(q_i-1) = p-1$, we know that $q_i-1 < p-1 \Rightarrow q_i < p$ for all primes dividing $d$.

First, let $p$ be the smallest prime dividing $m$. Let $K$ be the subset of $G$ consisting of those elements whose coefficient in $\varepsilon(G,1)$ is equal to $\lambda(p)\beta(m/p)/m$. As above, $L_p \subseteq K$. On the other hand, if any other layer $L_d\subseteq K$, then this implies that $\beta(d) = \beta(p)$, but by the previous paragraph all the prime divisors of $d$ are smaller than $p$, which is a contradiction. Therefore, $K = L_p$, which implies that $\overline{L_p} \in S$. For induction, suppose that if $p$ is a prime divisor of $m$ which is smaller than $k$ then $\overline{L_p}\in S$. Let $p$ be the smallest prime divisor of $m$ which is greater than or equal to $k$. Again, let $K$ be the subset of $G$ whose coefficient in $\varepsilon(G,1)$ is equal to $\lambda(p)\beta(m/p)/m$. Clearly, $L_p\subseteq K$. If $L_d\subseteq K$ for some other divisor $d$ of $m$, then $d = \prod_{i=1}^s q_i$, where $q_i$ is a prime divisor of $m$ strictly smaller than $p$. By our induction hypothesis, $\overline{L_{q_i}} \in S$ for all divisors of $d$. Furthermore, $\overline{G_{q_i}} \in S$ for all $i$ and hence $\overline{G_{d'}} \in S$ for all $d'\mid d$. Taking differences, this implies that $\overline{L_d}\in S$. So instead, we may set $K$ to be the subset of $G$ whose coefficient in $\varepsilon(G,1)- \frac{\lambda(p)\beta(m/p)}{m}\overline{L_d}$ is equal to $\lambda(p)\beta(m/p)/m$. Repeating this process finitely many times if necessary, eventually we will have that $K = L_p$, which implies that $\overline{L_p}\in S$. Therefore, by induction, $\overline{L_p}\in S$ for all $p\bmid |G|$. This implies that $G_p = \langle L_p\rangle$ is an $S$-subgroup.
\end{proof}

\begin{Lem}\label{lem:ScycIdemTrivial2} Let $S$ be a Schur ring over a $\QG$ where $G$ is a cyclic group. Let $H\trianglelefteq G$. If $\varepsilon(G,H) \in S$, then $\overline{M} \in S$ for all $M\in \M(G,H)$. \end{Lem}

\begin{proof}
Now, $\Stab(\varepsilon(G,H)) = H$, which implies $\overline{H}\in S$, by \lemref{prop:Stab}. Therefore, the result follows from  \corref{cor:HadamardMap}, \corref{cor:IdemDescent}, and \lemref{lem:ScycIdemTrivial}.
\end{proof}

\begin{Thm}\label{thm:ScycIdem} Let $G$ be a finite cyclic group and let $S$ be a Schur ring over $\QG$. Then $\varepsilon(S,H)$ is primitive for all $\overline{H} \in S$. In particular, $\{\varepsilon(S,H)\mid \overline{H}\in S\}$ is a complete set of primitive idempotents in $S$.
\end{Thm}

\begin{proof}
The proof is by induction on $|G|$. If $|G| = p$, a prime, then the lattice of $S$-subgroups is $\{1,G\}$, the entire lattice of subgroups. Thus, $\varepsilon(S,1) = \varepsilon(G,1)$ and $\varepsilon(S,G) = \widehat{G}$, which are primitive by \corref{thm:CyclicIdempotents}. Next, suppose that the result holds for all cyclic groups with order less than $n$. Let $G = Z_n$ and let $\overline{H}\in S$. Then consider $\varepsilon(S,H)$. By \lemref{thm:IdempotentSum} and \lemref{thm:IdemDescent}, if $\pi : G \to G/H$ is the quotient map, then $\varepsilon(S,H)$ is primitive if and only if $\varepsilon(\pi(S), H/H)$ is primitive, where the latter is primitive by our induction hypothesis. Thus, it suffices to prove the case for $\varepsilon(S,1)$. 

Suppose that \begin{equation}\label{eq:ScycIdem} \varepsilon(S,1) = \varepsilon_1 + \varepsilon_2\end{equation} decomposes as a sum of nonzero idempotents. By \lemref{thm:IdempotentSum}, $\varepsilon(S,1)$ is a sum of primitive idempotents of the form $\varepsilon(G,H)$, where $H$ does not contain a minimal $S$-subgroup. So, \eqref{eq:ScycIdem} partitions this collection of primitive idempotents. We may assume that $\varepsilon(G,1)$ is involved in $\varepsilon_1$. Suppose that $\varepsilon_1 = \varepsilon(G,1) \in S$. Then by \lemref{lem:ScycIdemTrivial}, $S$ contains all the minimal subgroups of $G$. In particular, $\varepsilon(S,1) = \varepsilon(G,1)$, by \lemref{thm:IdempotentSum}, and is primitive by \corref{thm:CyclicIdempotents}. So, we may assume that $\varepsilon_1$ involves some other primitive idempotent $\varepsilon(G,H)$, with $H\neq 1$. 

Next, suppose that $\varepsilon(G,K)$ is involved in $\varepsilon_2$ and suppose that $H\cap K \neq 1$. Let $\pi : G \to G/(H\cap K)$ be the quotient map.  Now, $H,K \in \nn(S,1)$, which implies that $H\cap K \in \nn(S,1)$. Then $\pi(\varepsilon(G,H)),\; \pi(\varepsilon(G,K)) \neq 0$, by \corref{cor:IdemDescent} and \lemref{lem:CycIdem}.  This means that $\pi(\varepsilon(S,1))$ is an imprimitive idempotent of $\pi(S)$. But $\pi(S)$ is a Schur ring by \thmref{thm:CayleySchur} and $\pi(\varepsilon(S,1)) = \varepsilon(\pi(S), 1)$ by \corref{cor:IdemDescent}. Thus, $\pi(\varepsilon(S,1))$ is primitive by our induction hypothesis, a contradiction.  Hence, $H\cap K = 1$ for all $\varepsilon(G,K)$ involved in $\varepsilon_2$. By this consideration, for all subgroups $1 <L \le K$, $\varepsilon(G,L)$ must be involved in $\varepsilon_2$ and for all subgroups $1< L \le H$, $\varepsilon(G,L)$ must be involved in $\varepsilon_1$. In particular, we may assume that $H$ and $K$ have prime order.  

Next, $\varepsilon(G,HK)$ cannot be involved in $\varepsilon_1$ since $HK\cap K \neq 1$ nor $\varepsilon_2$ since $HK\cap H \neq 1$. Thus, $\varepsilon(G,HK)$ is not involved in $\varepsilon(S,1)$, which implies that $HK$ contains a minimal $S$-subgroup. But the only nontrivial subgroups of $HK$ are $H$, $K$, and $HK$, by order considerations. Thus, $HK$ must be a minimal $S$-subgroup, that is, $\overline{HK} \in S$. 

If $\mathcal{K} = \{K_\alpha\mid \varepsilon(G,K_\alpha)\text{ is involved in } \varepsilon_2\}$ and $\bigcap \mathcal{K} = K \neq 1$, then $\Stab(\varepsilon_2) = K$, which implies that $\overline{K}\in S$. This contradicts \lemref{thm:IdempotentSum}, since $K\in \nn(S,1)$.  So, $\varepsilon_2$ must involve at least two distinct primitive idempotents $\varepsilon(G,K_1)$ and $\varepsilon(G,K_2)$ and we may assume that both $K_1$ and $K_2$ have prime orders. Using the previous argument, $\overline{HK_1}, \overline{HK_2}\in S$. But then $\overline{HK_1}\circ \overline{HK_2}  = \overline{H} \in S$, by the distributivity of the lattice of subgroups of $G$. But this contradicts \lemref{thm:IdempotentSum}. Therefore, $\varepsilon(S,1)$ is primitive in $S$.
\end{proof}

\begin{Cor}\label{cor:perlisSchur} Let $S$ be a Schur ring over $\QG$ and let $\varepsilon \in S$ be an idempotent, with $G$ cyclic. Then $\varepsilon\in \Span_\Q\{\overline{H} \mid \overline{H} \in S\}$.\end{Cor}

We now will compute a few examples to illustrate. Let $G = Z_{12} = \langle z\rangle$. Then the six normal subgroups of $G$ are $G_1 = 1$, $G_2 = \langle z^6\rangle$,   $G_3 = \langle z^4\rangle$, $G_4 = \langle z^3\rangle$, $G_6 = \langle z^2\rangle$, and $G_{12} = G$, and the six primitive idempotents of $\Q[Z_{12}]$ are 
\begin{eqnarray*}
\varepsilon(G,1) &=&\frac{1}{3} - \frac{1}{3}z^6 - \frac{1}{6}(z^4+z^8) + \frac{1}{6}(z^2+z^{10})\\
\varepsilon(G,G_2) &=& \frac{1}{6}(1+z^6) - \dfrac{1}{6}(z^3+z^9)  -\frac{1}{12}(z^2+z^4+z^8+z^{10}) + \frac{1}{12}(z+z^5+z^7+z^{11})\\
\varepsilon(G,G_3) &=& \frac{1}{6}(1+z^4+z^8) - \dfrac{1}{6}(z^2+z^6+z^{10})\\
\varepsilon(G,G_4) &=& \frac{1}{6}(1+z^3+z^6+z^9) - \dfrac{1}{12}(z+z^2+z^4+z^5+z^7+z^8+z^{10}+z^{11})\\
\varepsilon(G,G_6) &=& \frac{1}{12}(1+z^2+z^4+z^6+z^8+z^{10}) - \dfrac{1}{12}(z+z^3+z^5+z^7+z^9+z^{11})\\
\varepsilon(G,G) &=& \frac{1}{12}(1+z+z^2+z^3+z^4+z^5+z^6+z^7+z^8+z^9+z^{10}+z^{11})
\end{eqnarray*}

As before, every subgroup of a cyclic group is characteristic, which implies that every subgroup  is an $S$-subgroup of every orbit Schur ring. Thus, the primitive idempotents of any orbit Schur ring are exactly the primitive idempotents of $\QG$. Consider 
\[S = \Span_\Q\{1, z^6, z^4+z^8, z^2+z^{10}, z+z^5+z^9, z^3+z^7+z^{11}\},\] which is not an orbit ring. But $S$ is a Schur ring over $G$ and its $S$-subgroups are $1$, $G_2$, $G_3$, $G_6$, and $G_{12}$. Therefore, the primitive idempotents of $S$ are 
\begin{eqnarray*}
\varepsilon(S,1) &=&\varepsilon(G,1)\\
\varepsilon(S,G_2) &=& \varepsilon(G,G_2) + \varepsilon(G,G_4) = \frac{1}{3}(1+z^6) - \frac{1}{6}(z^2+z^4+z^8+z^{10})\\
\varepsilon(S,G_3) &=& \varepsilon(G,G_3)\\
\varepsilon(S,G_6) &=& \varepsilon(G,G_6)\\
\varepsilon(S,G) &=& \varepsilon(G,G).
\end{eqnarray*} We have used \lemref{thm:IdempotentSum} to decompose each idempotent into a sum of primitive idempotents over $\QG$. We note that $\varepsilon(G,G_2) \notin S$ since the coefficients of $z^9$ and $z$ differ. Likewise, $\varepsilon(G,G_4)\notin S$. Thus, $\varepsilon(S,G_2)$ is primitive in $S$. 

For another example, consider the Schur ring $T$:
\[T = \Span_\Q\{1, z^6, z^4+z^{10}, z^2+z^{8}, z+z^3+z^5 + z^7+z^9+z^{11}\}.\] Then the $T$-subgroups are $1$, $G_2$, $G_6$, and $G_{12}$ and the primitive idempotents are
\begin{eqnarray*}
\varepsilon(T,1) &=&\varepsilon(G,1) + \varepsilon(G,G_3) = \frac{1}{2} - \frac{1}{2}z^6\\
\varepsilon(T,G_2) &=& \varepsilon(G,G_2) + \varepsilon(G,G_4)\\
\varepsilon(T,G_6) &=& \varepsilon(G,G_6)\\
\varepsilon(T,G) &=& \varepsilon(G,G).
\end{eqnarray*}  Since $\varepsilon(G,1), \varepsilon(G,G_3), \varepsilon(G,G_2), \varepsilon(G,G_4) \notin T$, $\varepsilon(T,G_1)$ and $\varepsilon(T,G_2)$ are primitive in $T$.

We present one last example. Consider the Schur ring $U$:
\[U = \Span_\Q\{1, z^4, z^8, z^2+z^6+z^{10}, z+z^5+z^9, z^3+z^7+z^{11}\}.\] Then the $U$-subgroups are $1$, $G_3$, $G_6$, and $G_{12}$ and the primitive idempotents are 
\begin{eqnarray*}
\varepsilon(U,1) &=&\varepsilon(G,1) + \varepsilon(G,G_2) + \varepsilon(G,G_4)\\
&=& \frac{2}{3} - \dfrac{1}{3}(z^4+z^8)\\
\varepsilon(U,G_3) &=& \varepsilon(G,G_3)\\
\varepsilon(U,G_6) &=& \varepsilon(G,G_6)\\
\varepsilon(U,G) &=& \varepsilon(G,G).
\end{eqnarray*}  Clearly, $\varepsilon(G,1)$, $\varepsilon(G,2)$, and $\varepsilon(G,G_4)\notin U$. Therefore, $\varepsilon(U,1)$ is primitive in $U$.


\section{Decomposition of Schur Rings over Cyclic Groups}\label{sec:DecompSchur}

In \thmref{thm:LatticeIdem}, we determined the Wedderburn decomposition of any lattice Schur ring. For cyclic groups, this decomposition characterizes lattice Schur rings.

\begin{Prop}\label{prop:latticeWedderburn} Let $G$ be a finite cyclic group and let $S$ be a Schur ring over $\QG$. Then $S\cong \bigoplus \Q$ if and only if $S = S(\L)$ for some lattice $\L$ of subgroups of $G$.
\end{Prop}

\begin{proof} If $S = S(\L)$, then $S\cong \bigoplus \Q$ by \thmref{thm:LatticeIdem}. Suppose that $S \cong \bigoplus \Q$. Then the complete set of primitive idempotents of $S$ forms a basis. But this set is $\{\varepsilon(S,H) \mid \overline{H}\in S\}$ by \thmref{thm:ScycIdem}, and $\Span\{\varepsilon(S,H) \mid \overline{H}\in S\}$ is a lattice Schur ring by \propref{prop:S=T}.
\end{proof}

Let $\K_n = \Q(\zeta_n)$. From Galois theory, we know there is a one-to-one correspondence between the subfields of the cyclotomic field $\K_n$ and the subgroups of the Galois group $\G(\K_n/\Q) $. Each of these subfields is the fixed subalgebra of some Galois subgroup and is spanned by sums of roots of unity. 

In more generality, let $A$ be an algebra over a field $F$ and let $\H \le \Aut_F(A)$ be finite, where $\Aut_F(A)$ is the group of $F$-algebra automorphisms of $A$. Then 
\[A^\H = \{\alpha \in A\mid \sigma(\alpha) = \alpha,\;\text{for all}\;\sigma\in \H\},\] and $A^\H$ is spanned by the orbit sums of $\H$. In fact, if $B$ is a basis of $A$, then $A^\H$ is spanned by the orbit sums of elements from $B$. In particular, orbit Schur rings are spanned by the orbit sums of elements of $G$.

Let $\omega_n : \Q[Z_n] \to \Q(\zeta_n)$ be the $\Q$-algebra map uniquely defined by the relation $\omega_n(z) = \zeta_n$. If $A$ is any subalgebra of $\Q[Z_n]$ containing $1$, then $\Q\subseteq \omega_n(A) \subseteq \K_n$. Thus, $\omega_n(A)$ is a subfield of $\K_n$. In particular, $\omega_n(S)$ is a subfield of $\K_n$ for any Schur ring over $Z_n$.

Every automorphism of $Z_n$ is determined by $z\mapsto z^m$, and every automorphism on $\K_n$ is determined by $\zeta\mapsto \zeta^m$, where $m$ is unique modulo $n$. Identifying these congruence classes provides an automorphism between $\Aut(Z_n)$ and $\G(\K_n/\Q)$. For this reason, we may identify the two groups as the same and denote this group as $\G_n$. In particular, if $\sigma \in \G_n$, then $\sigma\circ \omega_n = \omega_n\circ \sigma$. Therefore, $\omega_n$ preserves the orbit structure of any subgroup of $\G_n$. This implies that the $\H$-periods, for $\H\le \G_n$, are preserved by $\omega_n$. 

\begin{Lem}\label{prop:RationalCyclicLattice} For each $\H \le \G_n$, $\omega_n(\Q[Z_n]^\H) = \K_n^\H$.
\end{Lem}

\begin{proof}
Using the map $\omega_n : \Q[Z_n] \to \K_n$, we have that the orbit sums of $\H$ in $Z_n$ map onto the orbit sums of $\H$ in $\K_n$. Thus, $\omega_n : \Q[Z_n]^\H \to \K_n^\H$ is  surjective.
\end{proof}


Let $\pi_d : G \to G/G_{n/d}$ be the natural quotient map. Let $\H\le \G_n$. Now, $\H$ can be viewed as a set of integers modulo $n$. Also, for any $d\mid n$,  $\H$ can be viewed as a set of integers modulo $d$ (with possible redundancies). Thus, we may use $\H$ to denote a subgroup of $\G_n$  and a subgroup of $\G_d$.

\begin{Lem}\label{thm:invariantAuto} Let $G = Z_n$. For each $\H\le \G_n$, $\pi_d(\QG^\H) = \Q[G_d]^\H$.
\end{Lem}

\begin{proof}
Using the map $\pi_d : \Q[G] \to \Q[G/G_{n/d}]$, the orbit sums of $\H$ in $G$ map onto the orbit sums of $\H$ in $G_d$ scaled possibly by a constant depending on the orbit. Thus, $\pi_d : \Q[G]^\H \to \Q[G/G_{n/d}]^\H$ is surjective.
\end{proof}

The next result  generalizes the theorem of Perlis and Walker to all orbit Schur rings.

\begin{Thm}\label{thm:Sperlis} Let $G = Z_n$, a cyclic group of order $n$, and let $\H \le \Aut(G)$. Then 
\[\QG^\H \cong \bigoplus_{d\mid n} \Q(\zeta_d)^\H.\] \end{Thm}

\begin{proof} 
Let $S = \QG^\H$. Since all subgroups of $G$ are characteristic, the primitive idempotents of $S$ are exactly the primitive idempotents of $\QG$, by \corref{thm:CyclicIdempotents}.  By the Perlis-Walker Theorem, each idempotent corresponds to a cyclotomic field and hence to a divisor of $n$. Let $\varepsilon_d$ denote the idempotent in $\QG$ such that $\QG\varepsilon_d \cong \K_d$. Let $\omega^{(d)} : \QG \to \K_d$ be the representation afforded by $\omega^{(d)}(z) = \zeta_d$ for each $d\mid n$. Using previous notation, we have $\omega^{(d)} = \omega_d\circ \pi_d$. For each element $x\in \K_d$, there exists some $\alpha \in \Q[G]$ such that $\omega^{(d)}(\alpha) = x$. Define $\varphi_d : \K_d \to \QG\varepsilon_d$ as $\varphi^{(d)}(x) = \alpha\varepsilon_d$. It is routine to check that $\varphi_d$ is an isomorphism. Then  
\[S = \bigoplus_{d\mid n} S\varepsilon_d.\]

 Now, the map $\varphi_d\circ \omega^{(d)}$ is multiplication by $\varepsilon_d$ and hence is the natural projection map $\QG \to \QG\varepsilon_d$. Thus, the restriction $\varphi_d\circ\omega^{(d)}: S\to S\varepsilon_d$ is the projection map.  Finally,
\begin{eqnarray*}
S\varepsilon_d &=& \varphi_d\circ \omega^{(d)}(S) = \varphi_d\circ \omega_d\circ \pi_d(S) = \varphi_d\circ \omega_d (\Q[G_d]^\H), \quad\text{by \lemref{thm:invariantAuto}},\\
&=& \varphi_d(\K_d^\H),\qquad\text{by \lemref{prop:RationalCyclicLattice}}.
\end{eqnarray*} But $\varphi_d|_{\K_d^\H}$ is injective and hence $\varphi_d : \K_d^\H \to S\varepsilon_d$ is an isomorphism, which finishes the proof.
\end{proof}

 In more generality, let $S$ be a Schur ring over a cyclic group. If $\overline{G_{n/d}} \in S$, then 
\[\omega^{(d)}(\varepsilon(S,G_{n/d})) = \omega_d\circ \pi_d(\varepsilon(S,G_{n/d})) = \omega_d(\varepsilon(\pi_d(S), 1)) = 1\in \K_d.\] In particular, Schur rings over cyclic groups decompose as sums of subfields of cyclotomic fields, where the degree of each cyclotomic field corresponds to the index of an $S$-subgroup. 

To illustrate this procedure, let $G = Z_{12}$ and let $S$, $T$, and $U$ be the Schur rings defined at the end of \secref{sec:IdemSchur}. 

First, $S$ has five primitive idempotents corresponding to the subgroups $G$, $G_6$, $G_3$, $G_2$, and $1$. Thus, $S$ has representations in $\Q$, $\K_2$, $\K_4$, $\K_6$, and $\K_{12}$. Since $\dim \Q = \dim \K_2 =1$,  $S\varepsilon(S,G) \cong S\varepsilon(S,G_6) \cong \Q$, as $\Q$-algebras. Since $\dim \omega^{(4)}(S) = \dim \K_3 = 2$, we have $S\varepsilon(S,G_3) \cong \K_3 = \Q(\zeta_3)$. Also, $\dim \omega^{(6)}(S) = 1$, we have $S\varepsilon(S,G_2) \cong \Q$. This accounts for five of the six dimensions of $S$. Hence, $S\varepsilon(S,1) = 1$. Therefore, 
\[S \cong \Q^4\oplus \Q(i).\]

Next, $T$ has four primitive idempotents corresponding to the subgroups $G$, $G_6$, $G_2$, and $1$, which gives representations in $\Q$, $\K_2$, $\K_6$, and $\K_{12}$, respectively. Like before, $T\varepsilon(T,G) \cong T\varepsilon(T,G_6) \cong \Q$. Since $\omega^{(6)}(T) = 2$, it must be that $T\varepsilon(T,G_2) \cong \K_6 \cong \Q(\zeta_3)$. Since $\dim T = 5$, it follows that $T\varepsilon(T,1) \cong \Q$. Therefore, 
\[T \cong \Q^3\oplus \Q(\zeta_3).\]

Lastly, $U$ has four primitive idempotents corresponding to $G$, $G_6$, $G_3$, and $G_1$. Thus, $U\varepsilon(U,G) \cong U\varepsilon(U,G_6) \cong \Q$. By dimension considerations, it must be that $\dim U\varepsilon(U,G_3) = \dim U\varepsilon(U,1) = 2$. This implies that $U\varepsilon(U,G_3) \cong \K_4$ and $U\varepsilon(U,1) \cong \K_3$ or $\K_4$. Since $\omega^{(12)}(z^4) = \zeta_3$, we have
\[U \cong \Q^2\oplus \Q(i)\oplus \Q(\zeta_3).\]

\section{Non-cyclic Examples}\label{sec:noncyclic}

Under what conditions can \thmref{thm:ScycIdem} be extended, that is, for a Schur ring $S$ with rational coefficients, when does its maximal lattice Schur subring contain all the central idempotents of $S$? We say that a Schur ring with rational coefficients is \emph{tidy} is all of its central idempotents are contained in the maximal lattice Schur subring. Tidiness is equivalent to saying that all the central idempotents are spanned by the normal $S$-subgroups. It is clear that if the maximal lattice subring contains all the central idempotents, then each central idempotent is spanned by the normal $S$-subgroups. To see the other direction, note that $\{\varepsilon(S,H)\mid \overline{H}\in S\}$ is a set of central idempotents and the span of $\{\varepsilon(S,H)\mid \overline{H}\in S\}$ is the same as the span of the normal $S$-subgroups. So, if each primitive central idempotent is spanned by the normal $S$-subgroups, then they are spanned by $\{\varepsilon(S,H)\mid \overline{H}\in S\}$, which means that $\{\varepsilon(S,H)\neq 0\mid \overline{H}\in S\}$ must be the set of primitive central idempotents of $S$. Thus, $S$ is tidy. 

By above, all lattice Schur rings are tidy and every Schur ring over a cyclic groups also is tidy. Also, all group algebras over abelian groups are tidy. But tidiness does not hold for Schur rings in general, that is, there exists Schur rings for which the primitive central idempotents are not spanned by the normal $S$-subgroups. In fact, a counterexample can be found among abelian groups. Let $G = Z_3\times Z_3 = \langle a, b\rangle$ and let \[S = \Span_\Q\{1, a+a^2+b+b^2, ab+a^2b^2+ab^2+a^2b\}.\] Then $S$ is an orbit Schur ring afforded by the automorphism subgroup generated by the automorphism $\sigma : a\mapsto b, b\mapsto a^2$. Now, the lattice of $S$-subgroups is simply $\{1, G\}$. Thus, $\varepsilon(S,1) = 1 - \widehat{G}$ and $\varepsilon(S,G) = \widehat{G}$. By \lemref{thm:IdempotentSum}, 
\[\varepsilon(S,1) = [\varepsilon(G, \langle a\rangle) + \varepsilon(G, \langle b\rangle)] + [\varepsilon(G, \langle ab\rangle) + \varepsilon(G, \langle ab^2\rangle)].\] But 
\begin{eqnarray*}
\varepsilon(G,\langle a\rangle) + \varepsilon(G,\langle b\rangle) &=& \frac{1}{3}(1+a+a^2) + \frac{1}{3}(1+b+b^2) - \frac{2}{9}\overline{G}\\
& =& \frac{2}{3} + \frac{1}{3}(a+a^2+b+b^2) - \frac{2}{9}\overline{G} \in S.
\end{eqnarray*}
 Similarly, $\varepsilon(G,\langle ab\rangle) + \varepsilon(G,ab^2)\in S$. Hence, $\varepsilon(S,1)$ is imprimitive decomposes as a sum of two nonzero idempotents in $S$. Therefore, $S$ is not tidy.  

The above example also shows that \propref{prop:latticeWedderburn} does not hold for abelian groups, since $S\cong \Q^3$ but $S$ is not a lattice Schur ring. Likewise, the above Schur ring  shows that \thmref{thm:Sperlis} does not extend to Schur rings over abelian groups, since $\Q^3 \cong \QG^\H \not\cong \bigoplus_{d\mid n} a_d\Q(\zeta_d)^\H \cong \Q^5$.

The above example can be generalized to all elementary abelian groups of order $p^2$, where $p$ is an odd prime. On the other hand, all Schur rings over the Klein 4-group are tidy. So there do exists non-cyclic abelian groups for which every Schur ring is tidy. For another example, all Schur rings over an abelian groups of order 8 are tidy. To see this, we first define the concept of a Schur homomorphism. 

\begin{Def} Let $S$ and $T$ be Schur rings over groups $G$ and $H$, respectively. A \emph{Schur homomorphism} is a linear map $\sigma : S \to T$ which preserves $\cdot$, $\circ$, and $*$. 
\end{Def}

Now, let $S$ be a Schur ring over $G$  and let $\alpha\in S$ be nonzero such that $\alpha\circ \alpha = \alpha$ and $\alpha^2 = c\alpha$ for some scalar $c$. Since $\alpha\circ \alpha = \alpha$, there exists some subset $H\subseteq G$ such that $\alpha = \overline{H}$. Since $\overline{H}^2 = c\overline{H}$, we have that $H$ is a nonempty, multiplicatively closed subset of $G$. Since $|G|<\infty$, we have $H$ is a subgroup of $G$. Therefore, these ring-theoretic identities characterize the $S$-subgroups of $S$. In particular, if $\sigma : S\to T$ is a Schur isomorphism, then there is a one-to-one correspondence between the $S$- and $T$-subgroups.  If additionally, both groups are abelian, then all $S$- and $T$-subgroups are necessarily normal and $\sigma$ induces a Schur isomorphism on the maximal lattice subrings of these two Schur rings. Since Schur isomorphisms are ring isomorphisms, they preserve the primitive central idempotents of the Schur ring. If these  primitive central idempotents are spanned by $S$-subgroups, a Schur isomorphism will preserve this property as well, that is, Schur isomorphism preserves tidiness.  Therefore, if $S$ is a Schur ring over an abelian group and is Schur isomorphic to a Schur ring over a cyclic group, then $S$ is tidy.

Now, $Z_4\times Z_2$ has 28 distinct Schur rings and $Z_2\times Z_2\times Z_2$ has 100. By the above paragraph, we need to only consider the Schur rings which are not Schur isomorphic to a Schur ring over $Z_8$. This leaves eight Schur isomorphism types among the two groups. Two of these Schur rings are the group algebras $\Q[Z_4\times Z_2]$ and $\Q[Z_2\times Z_2\times Z_2]$, but abelian group algebras are tidy by \propref{thm:AbelianIdempotents}. Four of the six remaining types are Schur isomorphic to lattice Schur rings, which are tidy. The remaining two types are Schur isomorphic to the wedge products $\Q[Z_4] \wedge \Q[Z_2\times Z_2]$ and $\Q[Z_2\times Z_2] \wedge \Q[Z_4]$. Both Schur rings have the same Wedderburn decomposition, which is $\Q^4\oplus \Q(i)$. Thus, both Schur rings contain five primitive idempotents. Both Schur rings contain six $S$-subgroups which produce five nonzero idempotents. Since these must be the primitive idempotents, the central idempotents of these Schur rings are spanned by the $S$-subgroups. Therefore, each Schur ring over an abelian group of order 8 is tidy.

Considering the examples from above, it is not yet clear for which abelian groups every Schur ring is tidy. Hopefully, future efforts will answer this question.

\underline{Acknowledgments}: The contents of this paper are part of the author's doctoral dissertation, written under the supervision of Stephen P. Humphries, that will be submitted to Brigham Young University. All computations made in preparation of this paper were accomplished using the computer software Magma \cite{Magma}.

\bibliographystyle{acm}
\bibliography{Srings}

\begin{thebibliography}{10}

\bibitem{Magma}
{\sc Bosma, W., Cannon, J., and Playoust, C.}
\newblock The {M}agma algebra system. {I}. {T}he user language.
\newblock {\em J. Symbolic Comput. 24}, 3-4 (1997), 235--265.
\newblock Computational algebra and number theory (London, 1993).

\bibitem{Rio07}
{\sc Broche, O., and del R\'{i}o, A.}
\newblock Wedderburn decomposition of finite group algebras.
\newblock {\em Finite Fields and Their Applications 13\/} (2007), 71--79.

\bibitem{Milies07}
{\sc Ferraz, R.~A., and Polcino~Milies, C.}
\newblock Idempotents in group algebras and minimal abelian codes.
\newblock {\em Finite Fields and Their Applications 13\/} (2007), 382--393.

\bibitem{Jespers96}
{\sc Goodaire, E.~G., Jespers, E., and Polcino~Milies, C.}
\newblock {\em Alternative Loop Rings}.
\newblock North-Holland Mathematics Studies, 184. Noth-Holland Publishing Co.,
  1996.

\bibitem{Jespers03}
{\sc Jespers, E., Leal, G., and Paques, A.}
\newblock Central idempotents in rational group algebras of finite nilpotent
  groups.
\newblock {\em J. Algebra Appl. 2}, no. 1, (2003).

\bibitem{Jespers12}
{\sc Jespers, E., Olteanu, G., and del R\'{i}o, A.}
\newblock Rational group algebras of finite groups: from idempotents to units
  of integral group rings.
\newblock {\em Algebr. Represent. Theory 15}, no. 2 (2012), 359--377.

\bibitem{Jespers13}
{\sc Jespers, E., Olteanu, G., and Van~Gelder, I.}
\newblock Group rings of finite strongly monomial groups: central units and
  primitive idempotents.
\newblock {\em J. Algebra 387\/} (2013), 99--116.

\bibitem{KlinPoschel}
{\sc Klin, M.~K., and P{o}schel, R.}
\newblock The k{o}nig problem, the isomorphism problem for cyclic graphs and
  the method of schur rings.
\newblock {\em Algebraic Methods in Graph Theory 1, 2\/} (1978).

\bibitem{Leung90}
{\sc Leung, K.~H., and Ma, S.~L.}
\newblock The structure of schur rings over cyclic groups.
\newblock {\em Journal of Pure and Applied Algebra 66\/} (1990), 287--302.

\bibitem{LeungII}
{\sc Leung, K.~H., and Man, S.~H.}
\newblock On schur rings over cyclic groups ii.
\newblock {\em Journal of Algebra 183\/} (1996), 273--285.

\bibitem{LeungI}
{\sc Leung, K.~H., and Man, S.~H.}
\newblock On schur rings over cyclic groups.
\newblock {\em Israel Journal of Mathematics 106\/} (1998), 251--267.

\bibitem{Ma}
{\sc Ma, S.~L.}
\newblock On association schemes, schur rings, strongly regular graphs and
  partial difference sets.
\newblock {\em Ars Combin. 21\/} (1989), 211--220.

\bibitem{Muzychuk93}
{\sc Muzychuk, M.~E.}
\newblock The structure of rational schur rings over cyclic groups.
\newblock {\em European Journal of Combinatorics 14\/} (1993), 479--490.

\bibitem{Muzychuk94}
{\sc Muzychuk, M.~E.}
\newblock On the structure of basic sets of schur rings over cyclic groups.
\newblock {\em Journal of Algebra 169\/} (1994), 655--678.

\bibitem{Rio03}
{\sc Olivieri, A., and del R\'{i}o, A.}
\newblock An algorithm to compute the primitive central idempotents and the
  wedderburn decomposition of a rational group algebra.
\newblock {\em J. Symbolic Comput. 35}, no. 6 (2003), 673--687.

\bibitem{Rio04}
{\sc Olivieri, A., del R\'{i}o, A., and Sim\'{o}n, J.~J.}
\newblock On monomial characters and central idempotents of rational group
  algebras.
\newblock {\em Comm. Algebra 32}, no. 4 (2004), 1531--1550.

\bibitem{Perlis}
{\sc Perlis, S., and Walker, G.}
\newblock Abelian group algebras of finite order.
\newblock {\em Trans. Amer. Math. Soc. 68\/} (1950), 420--426.

\bibitem{Toth}
{\sc T\'{o}th, L.}
\newblock A survey of the alternating sum-of-divisors function.
\newblock {\em ArXiv {\tt arXiv:1111.4842 [math.NT]}\/} (2011).

\bibitem{Gelder11}
{\sc Van~Gelder, I., and Olteanu, G.}
\newblock Finite group algebras of nilpotent groups: a complete set of
  orthogonal primitive idempotents.
\newblock {\em Finite Fields Appl. 17}, no. 2 (2011), 157--165.

\bibitem{Wielandt64}
{\sc Wielandt, H.}
\newblock {\em Finite Permutation Groups}.
\newblock Academic Press, New York-London, 1964.

\end{thebibliography}

\end{document}